\documentclass[11pt]{article}
\usepackage{amssymb,amsmath}
\usepackage[mathscr]{eucal}
\usepackage[cm]{fullpage}
\usepackage[english]{babel}
\usepackage[latin1]{inputenc}
\usepackage{color}

\usepackage[usenames,dvipsnames]{xcolor} 
\usepackage{graphicx} 
\usepackage{draftcopy} 
\draftcopyName{WATERMARK}{100} 
\usepackage{watermark} 
\definecolor{verylightblue}{rgb}{.855,.89,1.0}
\definecolor{lightbluegray}{rgb}{.788,.847,0.887}
\def\watermarktext{\put(-10,-680){
\rotatebox{
90}{\Huge\bf\color{lightbluegray}
{\scriptsize arXiv (v1.submitted.24092014)}
}}}
\watermark{%
\watermarktext 
} 

\def\dom{\mathop{\mathrm{Dom}}\nolimits}

\def\N{\mathbb N}

\def\POI{\mathcal{POI}}
\def\DP{\mathcal{DP}}
\def\ODP{\mathcal{ODP}}
\def\PODI{\mathcal{PODI}}

\def\PO{\mathcal{PO}}

\def\I{\mathcal{I}}

\newtheorem{theorem}{Theorem}[section]

\newtheorem{lemma}[theorem]{Lemma}

\newcommand{\NR}{{N\!\!R}}

\newenvironment{proof}{\begin{trivlist}\item[\hskip%
\labelsep{\bf Proof.}]}%
{\qed\rm\end{trivlist}}

\newcommand{\qed}{{\unskip\nobreak
\hfil\penalty50\hskip .001pt \hbox{}
          \nobreak\hfil
         \vrule height 1.2ex width 1.1ex depth -.1ex
           \parfillskip=0pt\finalhyphendemerits=0\medbreak}}

\newcommand{\lastpage}{\addresss}

\newcommand{\addresss}{\small \sf  
\noindent{\sc V\'\i tor H. Fernandes}, 
Departamento de Matem\'atica, 
Faculdade de Ci\^encias e Tecnologia, 
Universidade Nova de Lisboa, 
Monte da Caparica, 
2829-516 Caparica, 
Portugal; 
also: 
Centro de \'Algebra da Universidade de Lisboa, 
Av. Prof. Gama Pinto 2, 
1649-003 Lisboa, 
Portugal; 
e-mail: vhf@fct.unl.pt

\medskip

\noindent{\sc Teresa M. Quinteiro}, 
Instituto Superior de Engenharia de Lisboa, 
Rua Conselheiro Em\'\i dio Navarro 1, 
1950-062 Lisboa, 
Portugal; 
also: 
Centro de \'Algebra da Universidade de Lisboa, 
Av. Prof. Gama Pinto 2, 
1649-003 Lisboa, 
Portugal;
e-mail: tmelo@adm.isel.pt 
}

\title{Presentations for monoids of finite partial isometries}

\author{V\'\i tor H. Fernandes\footnote{This work was developed within the FCT Project PEst-OE/MAT/UI0143/2014 of CAUL, FCUL, and of Departamento de Matem\'atica da Faculdade de Ci\^encias e Tecnologia da Universidade Nova de Lisboa.}~ and 
Teresa M. Quinteiro\footnote{This work was developed within the FCT Project 
PEst-OE/MAT/UI0143/2014 of CAUL, FCUL, and of Instituto Superior de Engenharia de Lisboa.}
}

\begin{document}

\maketitle

\begin{abstract}
In this paper we give presentations for the monoid $\DP_n$ of all partial isometries on $\{1,\ldots,n\}$ and for its submonoid $\ODP_n$ of all order-preserving partial isometries. 
\end{abstract}

\medskip



\section*{Introduction}

Semigroups of order-preserving transformations have long been
considered in the literature. A short, and by no means comprehensive, history follows. 
In  1962, A\v\i zen\v stat
\cite{Aizenstat:1962} and Popova \cite{Popova:1962} exhibited presentations for
${\mathcal{O}} _n$, the monoid of all order-preserving full transformations
on a chain with $n$ elements, and for $\PO_n$, the  monoid of all
order-preserving partial transformations on a chain with $n$
elements. Some years later, in 1971, Howie \cite{Howie:1971}
studied some combinatorial and algebraic properties of ${\mathcal{O}} _n$
and, in 1992, Gomes and Howie \cite{Gomes&Howie:1992}
revisited the monoids ${\mathcal{O}} _n$ and $\PO_n$. Certain classes of divisors of the monoid ${\mathcal{O}} _n$ 
were determined by Higgins \cite{Higgins:1995} in 1995 and by Fernandes \cite{Fernandes:1997} in 1997.  
More recently, Laradji and Umar \cite{Laradji&Umar:2004,Laradji&Umar:2006} 
presented more combinatorial properties of these two monoids.  
The injective counterpart of ${\mathcal{O}} _n$, i.e. the monoid $\POI_n$ of all
injective members of $\PO_n$, has been object of study by the
first author in several papers
\cite{Fernandes:1997,Fernandes:1998,Fernandes:2001,Fernandes:2002,Fernandes:2008}, by Derech in \cite{Derech:1991}, 
by Cowan and Reilly in \cite{Cowan&Reilly:1995}, by Ganyushkin and Mazorchuk in  \cite{Ganyushkin&Mazorchuk:2003}, among other authors. 
Presentations for the monoid $\POI_n$ and for its extension $\PODI_n$,  
the monoid of all injective order-preserving or order-reversing
partial transformations on a chain with $n$ elements, 
were given by Fernandes \cite{Fernandes:2001} in 2001 and 
by Fernandes et al.~\cite{Fernandes&Gomes&Jesus:2004} in 2004, respectively.  
See also \cite{Fernandes:2002survey}, for a survey on known presentations of transformations monoids. 
We notice that the first author together with Delgado
\cite{Delgado&Fernandes:2000,Delgado&Fernandes:2004} have computed the abelian kernels of
the monoids $\POI_n$ and $\PODI_n$, by using a method that is
strongly dependent of given presentations of the monoids. 

\smallskip 

The study of semigroups of finite partial isometries was initiated by Al-Kharousi et al.~in 
\cite{AlKharousi&Kehinde&Umar:2014,AlKharousi&Kehinde&Umar:2014s}. 
The first of these two papers is dedicated to investigate some combinatorial properties of 
the monoid $\DP_n$ of all partial isometries on $\{1,\ldots,n\}$ and of its submonoid $\ODP_n$ of all order-preserving (considering the usual order of $\N$) partial isometries, in particular, their cardinalities. The second one presents the study of some of their algebraic properties, 
namely Green's structure and ranks. Recall that the \textit{rank} of a 
monoid $M$ is the minimum of the set $\{|X|\mid \mbox{$X\subseteq M$ and $X$ generates $M$}\}$. 
See e.g. \cite{Howie:1995} for basic notions on Semigroup Theory. 

\smallskip

The main aim of this paper is to exhibit presentations for the monoids $\DP_n$ and $\ODP_n$. 
We would like to point out that we made considerable use of computational tools, namely, of GAP \cite{GAP4}.

\medskip

Next, we introduce precise definitions of the objects considered in this work.  

Let $n\in\N$ and $X_n=\{1,\ldots,n\}\subset\N$ (with the usual arithmetic and order). 
Let $\I_n$ be the symmetric inverse semigroup on
$X_n$, i.e. the monoid, under composition of maps, of all partial permutations of $X_n$.

Let $\alpha\in\I_n$. We say that $\alpha$ is \textit{order-preserving}
(respectively, \textit{order-reversing}) if,
for all $i,j \in \dom(\alpha)$, $i\leq j$ implies $i\alpha\leq j\alpha$
(respectively, $i\alpha\geq j\alpha$).
Clearly, the product of two order-preserving transformations
or two order-reversing transformations is an
order-preserving transformation and the product of an order-preserving
transformation by an order-reversing transformation, or vice-versa, is
an order-reversing transformation. On the other hand, we say that $\alpha$ is an  
\textit{isometry} if, for all $i,j \in \dom(\alpha)$, $|i\alpha-j\alpha|=|i-j|$. 

Define
$$
\PODI_n=\{\alpha\in\I_n\mid\mbox{$\alpha$ is order-preserving or order-reversing}\}, 
$$
$$
\POI_n=\{\alpha\in\I_n\mid\mbox{$\alpha$ is order-preserving}\}, 
$$
$$
\DP_n=\{\alpha\in\I_n\mid\mbox{$\alpha$ is an isometry}\} 
$$
and 
$$
\ODP_n=\{\alpha\in\I_n\mid\mbox{$\alpha$ is an order-preserving isometry}\}.  
$$
All these sets are inverse submonoids of $\I_n$ (see \cite{Fernandes&Gomes&Jesus:2004,AlKharousi&Kehinde&Umar:2014s}). 
Obviously, $\POI_n\subseteq\PODI_n$ and $\ODP_n=\DP_n\cap\POI_n$.  
Moreover, as observed by  Al-Kharousi et al.~\cite{AlKharousi&Kehinde&Umar:2014s}, 
we also have $\DP_n\subseteq\PODI_n$. 

\smallskip 

For simplicity, from now on we consider $n\ge3$.

\section{Preliminaries}\label{presection}

Let $X$ be a set and denote by $X^*$ the free monoid generated by
$X$. A \textit{monoid presentation} is an ordered pair $\langle X\mid
R\rangle$, where $X$ is an alphabet and $R$ is a subset of
$X^*\times X^*$. An element $(u,v)$ of $X^*\times X^*$ is called a
{\it relation} and it is usually represented by $u=v$. To avoid
confusion, given $u, v\in X^*$, we will write $u\equiv v$, instead
of  $u=v$, whenever we want to state precisely that $u$ and $v$
are identical words of $X^*$. A monoid $M$ is said to be {\it
defined by a presentation} $\langle X\mid R\rangle$ if $M$ is
isomorphic to $X^*/\rho_R$, where $\rho_R$ denotes the smallest
congruence on $X^*$ containing $R$. For more details see
\cite{Lallement:1979} or \cite{Ruskuc:1995}.

Given a finite monoid $T$, it is clear that we can always exhibit
a presentation for it, at worse by enumerating all its elements,
but clearly this is of no interest, in general. So, by finding a
presentation for a finite monoid, we mean to find in some sense a
\textit{nice} presentation (e.g. with a small number of generators and
relations).

A usual method to find a presentation for a finite monoid
is described by the following result, adapted for the monoid case from
\cite[Proposition 3.2.2]{Ruskuc:1995}. 

\begin{theorem}[Guess and Prove method] \label{ruskuc} 
Let $M$ be a finite monoid, let $X$ be a generating set for $M$,
let $R\subseteq X^*\times X^*$ be a set of relations, and let
$W\subseteq X^*$. Assume that the following conditions are
satisfied:
\begin{enumerate}
\item The generating set $X$ of $M$ satisfies all the relations from $R$;
\item For each word $w\in X^*$, there exists a word $w'\in W$ such
that the relation $w=w'$ is a consequence of $R$;
\item $|W|\le|M|$.
\end{enumerate}
Then, $M$ is defined by the presentation $\langle X\mid R\rangle$.
\end{theorem}

Notice that, if $W$ satisfies the above conditions then, in fact,
$|W|=|M|$.

\smallskip

Let $X$ be an alphabet, $R\subseteq X^*\times X^*$ a set of
relations and $W$ a subset of $X^*$. We say that $W$ is a set of
\textit{forms} for the presentation $\langle X\mid R\rangle$ if
the condition 2 of Theorem \ref{ruskuc} is satisfied. Suppose
that the empty word belongs to $W$ and, for each letter $x\in X$ and
for each word $w\in W$, there exists a word $w'\in W$ such that
the relation $wx=w'$ is a consequence of $R$. Then, it is easy to
show that $W$ is a set of forms for $\langle X\mid R\rangle$.

\medskip

Given a presentation for a monoid, another method to find a new
presentation consists in applying Tietze transformations. For a
monoid presentation $\langle A\mid R\rangle$, the 
four \emph{elementary Tietze transformations} are:

\begin{description}
\item(T1)
Adding a new relation $u=v$ to $\langle A\mid R\rangle$,
providing that $u=v$ is a consequence of $R$;
\item(T2)
Deleting a relation $u=v$ from $\langle A\mid R\rangle$,
providing that $u=v$ is a consequence of $R\backslash\{u=v\}$;
\item(T3)
Adding a new generating symbol $b$ and a new relation $b=w$, where
$w\in A^*$;
\item(T4)
If $\langle A\mid R\rangle$ possesses a relation of the form
$b=w$, where $b\in A$, and $w\in(A\backslash\{b\})^*$, then
deleting $b$ from the list of generating symbols, deleting the
relation $b=w$, and replacing all remaining appearances of $b$ by
$w$.
\end{description}

The next result is well-known (e.g. see \cite{Ruskuc:1995}): 

\begin{theorem} 
Two finite presentations define the same monoid if and only if one
can be obtained from the other by a finite number of elementary
Tietze transformations $(T1)$, $(T2)$, $(T3)$ and $(T4)$.  
\end{theorem}

\medskip

As for the rest of this section, we will describe a process to obtain a
presentation for a finite monoid $T$ given a presentation for a certain
submonoid of $T$. This method was developed by Fernandes et al.~in \cite{Fernandes&Gomes&Jesus:2004} 
and applied there to construct a presentation, for instance, for the monoid $\PODI_n$. Here we will apply it to deduce a presentation for $\DP_n$.  

\medskip

Let $T$ be a (finite) monoid, $S$ be a submonoid of $T$ and $y$ an
element of $T$ such that $y^2=1$. Let us suppose that $T$ is
generated by $S$ and $y$. Let $X=\{x_1,\ldots,x_k\}$ ($k\in\N$) be
a generating set of $S$ and $\langle X\mid R\rangle$ a
presentation for $S$. Consider a set of forms $W$ for $\langle
X\mid R\rangle$ and suppose there exist two subsets $W_{\alpha}$ and
$W_{\beta}$ of $W$ and a word $u_0\in X^*$
such that $W=W_{\alpha}\cup W_{\beta}$ and
$u_0$ is a factor of each word in $W_{\alpha}$. Let $Y=X\cup\{y\}$
(notice that $Y$ generates $T$) and suppose that there exist words
$v_0,v_1,\ldots,v_k\in X^*$ such that the following relations over
the alphabet $Y$ are satisfied by the generating set $Y$ of $T$:
\begin{description}
\item $(\NR_1)$ $yx_i=v_iy$, for all $i\in\{1,\ldots,k\}$;
\item $(\NR_2)$ $u_0y=v_0$.
\end{description}

Observe that the relation (over the alphabet $Y$)
\begin{description}
\item $(\NR_0)$  $y^2=1$
\end{description}
is also satisfied (by the generating set $Y$ of $T$), by
hypothesis.

Let
$$
\mbox{$\overline R=R\cup \NR_0\cup \NR_1\cup \NR_2\;\;$ and
$\;\;\overline{W}=W\cup\{wy\mid w\in W_{\beta}\}\subseteq Y^*$.}
$$

Under these conditions, Fernandes et al.~\cite[Theorem 2.4]{Fernandes&Gomes&Jesus:2004}, proved: 

\begin{theorem} \label{overpresentation}
If $W$ contains the empty word then $\overline{W}$
is a set of forms for the presentation 
$\langle Y\mid \overline{R}\rangle$. Moreover, if $|\overline{W}|\leq|T|$ then
the monoid $T$ is defined by the presentation 
$\langle Y\mid\overline{R}\rangle$. 
\end{theorem}

\section{Some properties of the monoids $\ODP_n$ and $\DP_n$}\label{nownproperties}

The cardinals (among other combinatorial properties) of the monoids $\ODP_n$ and $\DP_n$ were computed by Al-Kharousi et al.~in 
\cite{AlKharousi&Kehinde&Umar:2014}. They showed that 
$$
|\ODP_n| = 3\cdot 2^n-2(n+1) \quad\text{and}\quad |\DP_n| = 3\cdot 2^{n+1}-(n+2)^2-1. 
$$
Next, we present another (short) proof of these equalities that, in our opinion,  gives us more insight. 

Let 
$
\I_n^+=\{\alpha \in \I_n \mid i\le i\alpha\}
$ (extensive partial permutations), 
$
\I_n^-=\{\alpha \in \I_n \mid i\alpha\le i\}
$ (co-extensive partial permutations), 
$
\ODP_n^+=\ODP_n\cap\I_n^+
$ 
and 
$
\ODP_n^-=\ODP_n\cap\I_n^-. 
$
On the other hand, it is easy to check that 
$$
\ODP_n=\{\alpha \in {\cal{I}}_n \mid  i\alpha-i=j\alpha-j, ~\text{for all}~ i,j \in\dom(\alpha)\}. 
$$
Then, clearly, 
$$
\ODP_n=\ODP_n^+\cup\ODP_n^-\quad\text{and}\quad 
E(\I_n)=\ODP_n^+\cap\ODP_n^-~, 
$$
where $E(\I_n)$ denotes the set of all idempotents of $\I_n$, which is formed by all partial identities of $X_n$. 
Furthermore, given a nonempty subset $X$ of $X_n$, in $\ODP_n$, we have exactly $\min(X)$ co-extensive transformations with domain $X$ and $n-\max(X)+1$ extensive transformations with domain $X$. 
For example, the elements of $\ODP_9$ with domain $\{3,5,6\}$ are 
$$
\binom{3~5~6}{1~3~4},~ \binom{3~5~6}{2~4~5},~ \binom{3~5~6}{3~5~6},~ 
\binom{3 ~5~6}{4~6~7},~ \binom{3~5~6}{5~7~8}~\text{and}~\binom{3~5~6}{6~8~9}
$$
(respectively, $3$ co-extensive and $4$ extensive transformations). 
Since the number of (nonempty) subsets of $X_n$ with minimum equal to $k$ is $2^{n-k}$ and the number of (nonempty) subsets of $X_n$ with maximum equal to $k$ is $2^{k-1}$, for $1\le k\le n$, we may deduce that 
$$
|\ODP_n^-|=1+\sum_{k=1}^{n}k2^{n-k}\quad\text{and}\quad 
|\ODP_n^+|=1+\sum_{k=1}^{n}(n-k+1)2^{k-1}. 
$$
On the other hand, it a routine to show that 
$
1+\sum_{k=1}^{n}k2^{n-k} = 2^{n+1}-(n+1) = 
1+\sum_{k=1}^{n}(n-k+1)2^{k-1},  
$ 
whence 
$$
|\ODP_n^-|=|\ODP_n^+| = 2^{n+1}-(n+1). 
$$
Finally, since $|E(\I_n)|=2^n$, we get 
$$
|\ODP_n| = |\ODP_n^-|+|\ODP_n^+| -|E(\I_n)| = 2(2^{n+1}-(n+1))-2^n 
=3\cdot 2^n-2(n+1). 
$$

Now, let 
$$
h=\left(\begin{array}{ccccc}
1&2&\cdots&n-1&n \\ 
n&n-1&\cdots&2&1
\end{array}\right) \in\DP_n.
$$
Notice that the identity (of $X_n$) and $h$ are the only permutations of $\DP_n$. 
On the other hand, given $\alpha\in \DP_n$, it is clear that $\alpha$ is an order-reversing transformation if and only if $h\alpha$ (and $\alpha h$) is an order-preserving transformation (see \cite{Fernandes&Gomes&Jesus:2004}). Hence, as $\alpha=h^2\alpha=h(h\alpha)$, it follows that the monoid $\DP_n$ is generated by $\ODP_n\cup \{h\}$. Moreover, it is also easy to deduce that 
$$
\DP_n=\ODP_n\cup h\cdot\ODP_n
\quad\text{and}\quad 
\ODP_n\cap h\cdot\ODP_n = \{\alpha\in\I_n\mid |\dom(\alpha)|\le1\}. 
$$ 
Thus 
$$
\begin{array}{rcl}
|\DP_n| & = & |\ODP_n|+ |h\cdot\ODP_n|- |\{\alpha\in\I_n\mid |\dom(\alpha)|\le1\}| \\
 & = & (3\cdot 2^n-2(n+1))+ (3\cdot 2^n-2(n+1)) -(n^2+1)\\
 & = & 3\cdot 2^{n+1}-(n+2)^2-1. 
\end{array}
$$

\medskip 

Next, we recall that Al-Kharousi et al.~proved in 
\cite{AlKharousi&Kehinde&Umar:2014s} that $\ODP_n$ is generated by 
$A=\{x_i \mid 1\leq i\leq n\}$ (as a monoid), where 
$$
x_i=
\left(\begin{array}{cccccc}
1&\cdots&n-i-1&n-i+1&\cdots&n \\ 
1&\cdots&n-i-1&n-i+1&\cdots&n
\end{array}\right), 
$$ 
for $1\leq i\leq n-2$, 
$$
x_{n-1}=\left(\begin{array}{ccccc}
1&2&\cdots&n-2&n-1 \\ 
2&3&\cdots&n-1&n
\end{array}\right)
\quad\text{and}\quad 
x_n=\left(\begin{array}{ccccc}
2&3&\cdots&n-1&n \\ 
1&2&\cdots&n-2&n-1
\end{array}\right). 
$$ 
Moreover, they also proved that $A$ is the unique minimal (for set inclusion) generating set of $\ODP_n$, from which follows immediately that $\ODP_n$ has rank equal to $n$ (as a monoid). 

Now, since $\ODP_n\cup\{h\}$ generates the monoid $\DP_n$, it follows that  $B=A\cup\{h\}$ is also a generating set of $\DP_n$. Moreover, it is easy to show that $x_{n-1}=hx_nh$ and $x_{n-i-1}=hx_ih$, for $1\leq i\leq n-2$. Therefore  
$C=\{h,x_n\}\cup\{x_i\mid 1\leq i\leq \lfloor \frac{n-1}{2} \rfloor\}$ is another generating set for $\DP_n$. Observe that $C$ as $\lfloor \frac{n+3}{2} \rfloor$ elements, which coincides with the rank of $\DP_n$ \cite{AlKharousi&Kehinde&Umar:2014s}.

\section{A presentation for $\ODP_n$} 

Consider the set $A=\{x_i \mid 1\leq i\leq n\}$ as an alphabet (with $n$ letters) and the set $R$ formed by the following $\frac{1}{2}n^2+\frac{1}{2}n+3$ monoid relations:

\begin{description}
\item[$(R_1)$] $x_i^2=x_i$,~ $1\leq i\leq n-2$;
\item[$(R_2)$] $x_ix_j=x_jx_i$,~ $1\leq i<j\leq n-2$;
\item[$(R_3)$] $x_{n-1}x_{i}=x_{i+1}x_{n-1}$,~ $1\leq i\leq n-3$;
\item[$(R_4)$] $x_nx_{i+1}=x_{i}x_n$,~ $1\leq i\leq n-3$;
\item[$(R_5)$] $x_{n-1}^2x_n=x_{1}x_{n-1}$;
\item[$(R_6)$] $x_nx_{n-1}^2=x_{n-1}x_{n-2}$; 
\item[$(R_7)$] $x_n^2x_{n-1}=x_{n-2}x_n$; 
\item[$(R_8)$] $x_{n-1}x_n^2=x_nx_{1}$;
\item[$(R_9)$] $x_{n-1}x_nx_{n-1}=x_{n-1}$;
\item[$(R_{10})$] $x_nx_{n-1}x_n=x_n$;
\item[$(R_{11})$] $x_n^n=x_{1}\cdots x_{n-2}x_{n-1}x_{n-2}$;
\item[$(R_{12})$] $x_n^{n+1}=x_n^n$. 
\end{description}

This section is dedicated to prove that $\langle A\mid R\rangle$ is a presentation for the monoid $\ODP_n$, using the method given by Theorem \ref{ruskuc}. 

\medskip 

Observe that we can easily deduce from $R$ the following four relations, which are useful to simplify some calculations: 
\begin{equation}\label{morerel}
x_{1}x_{n-1}^2=x_{n-1}^2,\quad x_{n-1}^2x_{n-2}=x_{n-1}^2, \quad
x_{n-2}x_n^2=x_n^2 \quad\text{and}\quad x_n^2x_{1}=x_n^2. 
\end{equation}

\medskip 

First, it is a routine matter to prove that all relations over the alphabet $A$ from $R$ are satisfied by the generating set $A$ of $\ODP_n$ (with the natural correspondence between letters and generators). 

\medskip 

Next, we define our set of forms for $\langle A\mid R\rangle$. 

Let 
$$
W_{n-1}=\left\{ x_1,\ldots,~x_n, ~x_{n-1}x_{n}, ~x_{n}x_{n-1}\right\}
$$
Notice that $|W_{n-1}|=n+2$. For $2\leq k\leq n-1$, let 

\begin{description}

\item $W_{n-k,1}=\left\{\left(x_{\ell_1}\cdots x_{\ell_{k-1}}\right)x_i \mid 1\le \ell_1<\cdots<\ell_{k-1}\le n-2, \quad \ell_{k-1}<i\leq n\right\}$, 

\item $W_{n-k,2}=\left\{\left(x_{\ell_1}\cdots x_{\ell_{k-1}}\right)x_{n-1}x_{n}, ~~\left(x_{\ell_1}\cdots x_{\ell_{k-1}}\right)x_{n}x_{n-1} \mid 1\le \ell_1<\cdots<\ell_{k-1}\le n-2\right\}$, 

\item $W_{n-2,3}=\left\{ x_{n-1}x_{n-2}x_n \right\}$ and, for $k\ge3$,  
$
W_{n-k,3}=\left\{\left(x_{\ell_1}\cdots x_{\ell_{k-2}}\right)x_{n-1}x_{n-2}x_n \mid 1\le \ell_1<\cdots<\ell_{k-2}\le n-2\right\}, 
$

\item $W_{n-k,4}=\left\{\left(x_{\ell_1}\cdots 
x_{\ell_{k-i}}\right)x_{n-1}^i, ~~x_{n-1}^k  
\mid i\le \ell_1<\cdots<\ell_{k-i}\le n-2,  
\quad 2\le i< k\right\}$, 

\item $W_{n-k,5}=\left\{\left(x_{\ell_1}\cdots 
x_{\ell_{k-i}}\right)x_{n}^i, ~~x_n^k \mid 1\le \ell_1<\cdots<\ell_{k-i}\le n-i-1,  \quad 2\le i< k\right\}$, 

\item $W_{n-k,6}=\left\{\left(x_{\ell_1}\cdots 
x_{\ell_{k-i}}\right)x_{n-1}^{i-1}x_{n-i}, ~~x_{n-1}^{k-1}x_{n-k} \mid 
i-1\le \ell_1<\cdots<\ell_{k-i}\le n-2, \quad 2\le i< k \right\}$, 

\item $W_{n-k,7}=\left\{\left(x_{\ell_1}\cdots 
x_{\ell_{k-i}}\right)x_{n}x_{1}x_n^{i-2}, ~~x_{n}x_{1}x_n^{k-2} \mid 
1\le \ell_1<\cdots<\ell_{k-i}\le n-i, \quad 2\le i< k\right\}$ 
\end{description}
and 
$$
W_{n-k}=W_{n-k,1}\cup W_{n-k,2}\cup W_{n-k,3}\cup 
W_{n-k,4}\cup W_{n-k,5}\cup W_{n-k,6}\cup W_{n-k,7}. 
$$ 
Notice that, for $2\leq k\leq n-1$, we have 
\begin{description}
\item $|W_{n-k,1}|=\binom{n-2}{k}+2\binom{n-2}{k-1}$, 

\item $|W_{n-k,2}|=2\binom{n-2}{k-1}$, 

\item $|W_{n-k,3}|=\binom{n-2}{k-2}$,  

\item $|W_{n-k,4}|=|W_{n-k,5}|=\sum_{i=2}^{k}\binom{n-i-1}{k-i}$, 

\item $|W_{n-k,6}|=|W_{n-k,7}|=\sum_{i=2}^{k}\binom{n-i}{k-i}$ 
\end{description}
and so $|W_{n-k}|= \binom{n-2}{k}+2\binom{n-2}{k-1} + 
2\binom{n-2}{k-1} +
\binom{n-2}{k-2} + 
2 \sum_{i=2}^{k}\binom{n-i-1}{k-i} + 
2 \sum_{i=2}^{k}\binom{n-i}{k-i} = 
\binom{n}{k} + 2 \sum_{i=1}^{k}\binom{n-i}{k-i}$.
Thus, for $1\leq k\leq n-1$, we have 
$$
|W_k| =  \binom{n}{n-k} + 2 \sum_{i=1}^{n-k}\binom{n-i}{n-k-i} =  \binom{n}{k} + 2 \sum_{i=0}^{n-k-1}\binom{k+i}{i} = \binom{n}{k} + 2 \sum_{i=k}^{n-1}\binom{i}{k}  
$$
and, by Gould (1.52) identity \cite[page 7]{Gould:1972}, i.e. 
$
\sum_{i=k}^{m}\binom{i}{k}=\binom{m+1}{k+1}, 
$
it follows 
$
|W_k|=\binom{n}{k}+2\binom{n}{k+1}. 
$
Finally, let 
$$
W_{0}=\left\{ x_n^n\right\} 
$$
and  
$$
W=\{1\}\cup\bigcup_{k=0}^{n-1} W_k
$$
Notice that 
$$
\begin{array}{rcl}
|W| & = & 2 + \sum_{k=1}^{n-1}|W_k|\\
 & = & 2+ \sum_{k=1}^{n-1}(\binom{n}{k}+2\binom{n}{k+1}) \\ 
 & = & 2+ \sum_{k=1}^{n-1}\binom{n}{k}+2\sum_{k=1}^{n-1}\binom{n}{k+1} \\ 
 & = & 2+ \sum_{k=1}^{n-1}\binom{n}{k}+2\sum_{k=2}^{n}\binom{n}{k} \\ 
 & = & 2 + (2^n-1-1)+2(2^n-1-n)\\
 & = & 3\cdot2^n-2(n+1)\\
 & = & |\ODP_n|. 
\end{array}
$$

Observe that, for $0\le k\le n-1$, each word of $W_k$ represents a transformation of rank $k$ of $\ODP_n$. 

\begin{lemma}\label{lema}
$W$ constitutes a set of forms for $\langle A\mid R\rangle$. 
\end{lemma} 
\begin{proof}
As observed after Theorem \ref{ruskuc} it suffices to show that 
for each letter $x\in A$ and
for each word $w\in W$, there exists a word $w'\in W$ such that
the relation $wx=w'$ is a consequence of $R$. 

In order to perform this aim, we consider separately the word $w\in W$ in each of the subsets considered above that defines $W$. Namely, in $W_{n-1}$, $W_{n-k,r}$, 
for $1\le r\le 7$ and $2\le k\le n-1$, and $W_0$. 

\begin{description}
\item[I.] Let $w\in W_{n-1}$ and $j\in\{1,\ldots,n\}$. We consider five cases. 
\begin{description}
\item[Case 1.] $w\equiv x_i$, with $i\in \{1,\ldots,n-2\}$. 
If $i<j$ then $wx_j\equiv x_ix_j\in W_{n-2,1}$. 
If $j<i$, by applying a relation $(R_2)$, 
we have $wx_j\equiv x_ix_j=x_jx_i \in W_{n-2,1}$. 
If $j=i$ then, by applying a relation $(R_1)$,  
we have $wx_j\equiv x_i^2=x_i \in W_{n-1}$. 

\item[Case 2.] $w\equiv x_{n-1}$. 
If $j\in \{1,\ldots,n-3\}$ then, 
by a relation $(R_3)$, 
we have $wx_j\equiv x_{n-1}x_j=x_{j+1}x_{n-1}\in W_{n-2,1}$. 
If $j=n-2$ then $wx_j\equiv x_{n-1}x_{n-2}\in W_{n-2,6}$.
If $j=n-1$ then $wx_j\equiv x_{n-1}^2 \in W_{n-2,4}$. 
If $j=n$ then $wx_j\equiv x_{n-1}x_n \in W_{n-1}$.

\item[Case 3.] $w\equiv x_n$. 
If $j=1$ then $wx_j\equiv x_{n}x_1 \in W_{n-2,7}$. 
If $j\in \{2,\ldots,n-2\}$ then, by a relation $(R_4)$, 
we have $wx_j\equiv x_{n}x_j=x_{j-1}x_{n}\in W_{n-2,1}$. 
If $j=n-1$ then $wx_j\equiv x_nx_{n-1} \in W_{n-1}$. 
If $j=n$ then $wx_j\equiv x_n^2 \in W_{n-2,5}$.

\item[Case 4.] $w\equiv x_{n-1}x_n$. 
If $j=1$ then, by relations $(R_8)$ and $(R_5)$,
we have $wx_j\equiv x_{n-1}x_nx_1=x_{n-1}^2x_n^2=x_1x_{n-1}x_n \in W_{n-2,2}$.
If $j\in \{2,\ldots,n-2\}$ then, by relations $(R_4)$ and $(R_3)$, 
we have $wx_j\equiv x_{n-1}x_nx_j=x_{n-1}x_{j-1}x_n=x_jx_{n-1}x_n\in W_{n-2,2}$. 
If $j=n-1$ then, by the relation $(R_9)$, 
we have $wx_j\equiv x_{n-1}x_nx_{n-1}=x_{n-1} \in W_{n-1}$. 
If $j=n$ then, by the relation $(R_8)$, 
we have $wx_j\equiv x_{n-1}x_n^2=x_nx_1 \in W_{n-2,7}$. 

\item[Case 5.] $w\equiv x_{n}x_{n-1}$. 
If $j\in \{1,\ldots,n-3\}$ then, by relations $(R_3)$ and $(R_4)$, 
we have $wx_j\equiv x_{n}x_{n-1}x_j=x_{n}x_{j+1}x_{n-1}=x_jx_{n}x_{n-1}\in W_{n-2,2}$.
If $j=n-2$ then, by relations $(R_6)$ and $(R_7)$,
we have $wx_j\equiv x_{n}x_{n-1}x_{n-2}=x_{n}^2x_{n-1}^2=x_{n-2}x_nx_{n-1} \in W_{n-2,2}$.
If $j=n-1$ then, by the relation $(R_6)$, 
we have $wx_j\equiv x_{n}x_{n-1}^2=x_{n-1}x_{n-2} \in W_{n-2,6}$.  
If $j=n$ then, by the relation $(R_{10})$, 
we have $wx_j\equiv x_nx_{n-1}x_n=x_n \in W_{n-1}$. 
\end{description}

\item[II.] Let $w\in W_{n-k,1}$, with $2\le k\le n-1$, 
and $j\in\{1,\ldots,n\}$. 
Then, for some  $1\le \ell_1<\cdots<\ell_{k-1}\le n-2$ and $\ell_{k-1}<i\leq n$, 
we have $w\equiv \left(x_{\ell_1}\cdots x_{\ell_{k-1}}\right)x_i$. 
We consider three  cases. 
\begin{description}
\item[Case 1.] $i\le n-2$. 
If $i<j$ then $wx_j\equiv \left(x_{\ell_1}\cdots x_{\ell_{k-1}}\right)x_ix_j\equiv \left(x_{\ell_1}\cdots x_{\ell_{k-1}}x_i\right)x_j\in W_{n-(k+1),1}$. 
If $j<i$ then $wx_j\equiv \left(x_{\ell_1}\cdots x_{\ell_{k-1}}\right)x_ix_j=\left(x_{\ell_1}\cdots x_{\ell_t}x_jx_{\ell_{t+1}}\cdots x_{\ell_{k-1}}\right)x_i$, with $\ell_t\le j\le\ell_{t+1}$ and $1\le t\le k-1$, by applying $k-t$ times relations $(R_2)$.
If either $\ell_t$ or $\ell_{t+1}$ is equal to $j$ then, by a relation $(R_1)$, 
we have $\left(x_{\ell_1}\cdots x_{\ell_t}x_jx_{\ell_{t+1}}\cdots x_{\ell_{k-1}}\right)x_i= w\in W_{n-k,1}$. 
Otherwise $\left(x_{\ell_1}\cdots x_{\ell_t}x_jx_{\ell_{t+1}}\cdots x_{\ell_{k-1}}\right)x_i\in W_{n-(k+1),1}$. 
If $j=i$ then $wx_j\equiv\left(x_{\ell_1}\cdots x_{\ell_{k-1}}\right)x_i^2=w \in W_{n-k,1}$, by applying a relation $(R_1)$. 

\item[Case 2.] $i=n-1$. 
If $j\in \{1,\ldots,n-3\}$ then 
$wx_j\equiv\left(x_{\ell_1}\cdots x_{\ell_{k-1}}\right)x_{n-1}x_j=
\left(x_{\ell_1}\cdots x_{\ell_{k-1}}\right)x_{j+1}x_{n-1}$, by applying a relation $(R_3)$. Now, as above, 
by applying enough times relations from $(R_2)$ and,
for $j+1\in \{\ell_1,\ldots,\ell_{k-1}\}$, also a relation $(R_1)$, we obtain 
$wx_j=w\in W_{n-k,1}$ or $wx_j=w'\in W_{n-(k+1),1}$. 
If $j=n-2$ then $wx_j\equiv \left(x_{\ell_1}\cdots x_{\ell_{k-1}}\right)x_{n-1}x_{n-2} \in W_{n-(k+1),6}$.
If $j=n-1$ and $\ell_1>1$ then 
$wx_j\equiv \left(x_{\ell_1}\cdots x_{\ell_{k-1}}\right)x_{n-1}^2\in W_{n-(k+1),4}$.
If $j=n-1$ and $\ell_1=1$ then 
$wx_j\equiv \left(x_{1}\cdots x_{\ell_{k-1}}\right)x_{n-1}^2=\left(x_{\ell_2}\cdots x_{\ell_{k-1}}\right)x_1x_{n-1}^2=\left(x_{\ell_2}\cdots x_{\ell_{k-1}}\right)x_{n-1}^2 \in W_{n-k,4}$, 
by first applying $k-1$ times relations from $(R_2)$ and then 
the first relation of (\ref{morerel}).
If $j=n$ then $wx_j\equiv \left(x_{\ell_1}\cdots x_{\ell_{k-1}}\right)x_{n-1}x_n\in  W_{n-k,2}$.

\item[Case 3.] $i=n$. 
If $j=1$ then $wx_j\equiv\left(x_{\ell_1}\cdots x_{\ell_{k-1}}\right)x_{n}x_1 \in W_{n-(k+1),7}$.
If $j\in \{2,\ldots,n-2\}$ then 
$wx_j\equiv\left(x_{\ell_1}\cdots x_{\ell_{k-1}}\right)x_{n}x_j=
\left(x_{\ell_1}\cdots x_{\ell_{k-1}}\right)x_{j-1}x_{n}$, by a relation $(R_4)$. 
Again, as above, 
by applying enough times relations from $(R_2)$ and, for $j-1\in \{\ell_1,\ldots,\ell_{k-1}\}$, also a relation $(R_1)$, we obtain 
$wx_j=w\in W_{n-k,1}$ or $wx_j=w'\in W_{n-(k+1),1}$. 
If $j=n-1$ then 
$wx_j\equiv \left(x_{\ell_1}\cdots x_{\ell_{k-1}}\right)x_nx_{n-1} \in  W_{n-k,2}$.
If $j=n$ and $\ell_{k-1}<n-2$ then 
$wx_j\equiv \left(x_{\ell_1}\cdots x_{\ell_{k-1}}\right)x_{n}^2
\in W_{n-(k+1),5}$. 
If $j=n$ and $\ell_{k-1}=n-2$ then 
$wx_j\equiv\left(x_{\ell_1}\cdots x_{\ell_{k-2}}x_{n-2}\right)x_{n}^2\equiv \left(x_{\ell_1}\cdots x_{\ell_{k-2}}\right) x_{n-2}x_{n}^2=\left(x_{\ell_1}\cdots x_{\ell_{k-2}}\right)x_{n}^2\in W_{n-k,5}$, by applying the third relation of (\ref{morerel}). 
\end{description}

\item[III.] For $w\in W_{n-k,r}$, with $2\le k\le n-1$ and $2\le r\le 7$, 
and $j\in\{1,\ldots,n\}$, similar calculations to the previous cases, mostly routine, assure us the existence of $w'\in W$ such that the relation $wx_j=w'$ is a consequence of $R$. In fact, we may find $w'$ belonging to:  
\begin{enumerate}
\item $W_{n-k,2}\cup W_{n-(k+1),2}\cup W_{n-k,1}\cup W_{n-(k+1),6}\cup W_{n-(k+1),7}$,~ if $r=2$; 

\item $W_{n-3,3}\cup W_{n-2,6}\cup W_{n-2,7}$,~ if $r=3$ and $k=2$; 

\item $W_{n-k,3}\cup W_{n-(k+1),3}\cup W_{n-k,6}\cup W_{n-k,7}$,~ if $r=3$ and $k\ge3$; 

\item $W_{n-k,4}\cup W_{n-(k+1),4}\cup W_{n-k,1}\cup W_{n-(k+1),6}$,~ if $r=4$; 

\item $W_{n-k,5}\cup W_{n-(k+1),5}\cup W_{n-k,1}\cup W_{n-k,4}\cup W_{n-(k+1),7}$,~ if $r=5$; 

\item $W_{n-k,6}\cup W_{n-(k+1),6}\cup W_{n-k,3}\cup W_0$,~ if $r=6$; 

\item $W_{n-k,7}\cup W_{n-(k+1),7}\cup W_{n-k,3}$,~ if $r=7$.  
\end{enumerate}

\item[IV.] Finally, we show that $x_n^nx_j=x_n^n$ is a consequence of $R$, for $j\in \{1,\ldots,n\}$. We consider four cases. 
\begin{description}
\item[Case 1.] $j\in \{1,\ldots,n-3\}$. Then 
$$
\begin{array}{rcl}
x_n^nx_j &=& x_1\cdots x_{n-2}x_{n-1}x_{n-2}x_j\\ 
&=&x_1\cdots x_{n-2}x_{n-1}x_jx_{n-2}\\
&=& x_1\cdots x_{n-2}x_{j+1}x_{n-1}x_{n-2}\\ 
&=& x_1\cdots x_{j+1}^2 \cdots x_{n-2}x_{n-1}x_{n-2}\\
&=& x_1\cdots x_{j+1} \cdots x_{n-2}x_{n-1}x_{n-2}\\
&=&x_n^n~, 
\end{array}
$$
by applying relations $(R_{11})$, $(R_2)$, $(R_3)$, $(R_2)$, $(R_1)$ and $(R_{11})$ in an orderly manner. 

\item[Case 2.]  $j=n-2$. By applying relations $(R_{11})$ and $(R_1)$, we have 
$$
x_n^nx_j=x_1\cdots x_{n-2}x_{n-1}x_{n-2}^2=x_1\cdots x_{n-2}x_{n-1}x_{n-2}=x_n^n. 
$$  

\item[Case 3.] $j=n-1$. We obtain 
$$
\begin{array}{rcl}
x_n^nx_j & = & x_n^nx_nx_{n-1}\\ 
&=& x_1\cdots x_{n-2}x_{n-1}x_{n-2}x_nx_{n-1}\\
&=& x_1\cdots x_{n-2}x_{n-1}x_n^2x_{n-1}^2\\
&=& x_1\cdots x_{n-2}x_{n-1}x_nx_{n-1}x_{n-2}\\
&=& x_1\cdots x_{n-2}x_{n-1}x_{n-2}\\
&=&x_n^n~, 
\end{array}
$$
by applying relations $(R_{12})$, $(R_{11})$, $(R_7)$, $(R_6)$, $(R_9)$ and $(R_{11})$ in an orderly manner. 

\item[Case 4.]  $j=n$. This is exactly relation $(R_{12})$. 
\end{description}
\end{description}
This completes the proof of the lemma. 
\end{proof}

At this stage, we proved all the conditions of Theorem \ref{ruskuc}. Therefore, we have: 

\begin{theorem}
The monoid $\ODP_n$ is defined by the presentation $\langle A\mid R \rangle$, on $n$ generators and $\frac{1}{2}n^2+\frac{1}{2}n+3$ relations.
\end{theorem}

\section{Presentations for $\DP_n$}

In this section we exhibit two presentations for $\DP_n$. We start by construct a presentation associated to the set of generators $B$ and then deduce a new one associated to the set of generators $C$. 
Recall that $B=A\cup \{h\}=\{x_1,\ldots,x_n,h\}$ and 
$C=\{h,x_n\}\cup\{x_i\mid 1\leq i\leq \lfloor \frac{n-1}{2} \rfloor\}$. 

Consider the set $B$ as an alphabet (with $n+1$ letters) and the set $\overline{R}$ formed by all relations from $R$ (defined in the previous section) together with the following monoid relations: 
\begin{enumerate}
\item[$(\NR_0)$] $h^2=1$; 
\item[$(\NR_1)$] $hx_{n-1}=x_{n}h$,~ $hx_n=x_{n-1}h$ ~and~ 
$hx_i=x_{n-i-1}h$,~ $1\leq i\leq n-2$; 
\item[$(\NR_2)$] $x_n^{n-1}h=x_nx_{n-1}x_{n-2}\cdots x_2x_1$. 
\end{enumerate}

Notice that $\overline{R}$ is a set of $\frac{1}{2}n^2+\frac{3}{2}n+5$ monoid relations over the alphabet $B$. 

\medskip 

It is a routine matter to prove that all relations from $\overline{R}$ are satisfied by the generating set $B$ of $\DP_n$ (with the natural correspondence between letters and generators). 

\medskip 

Now, let $\alpha_{i,j}=\binom{i}{j}\in\ODP_n$, for $1\le i,j\le n$. 
Then $x_n^{n-1}=\alpha_{n,1}$ and 
$
\alpha_{i,j}=\alpha_{i,n}\alpha_{n,1}\alpha_{1,j}=\alpha_{i,n}x_n^{n-1}\alpha_{1,j}~,
$
for $1\leq i,j\leq n$.
Let $u_i$ and $v_i$ be the words of $W$ (defined in the previous section or even any other set of representatives of all elements of $\ODP_n$ over the alphabet $A$) that represents the elements $\alpha_{i,n}$ and $\alpha_{1,i}$ of $\ODP_n$, respectively, 
for $1\leq i\leq n$. 
Hence the word $u_ix_n^{n-1}v_j$ over the alphabet $A$ also represents the transformation $\alpha_{i,j}\in\ODP_n$, for $1\le i,j\le n$. 

Let $W_{\alpha}=\{u_ix_n^{n-1}v_j \mid 1\leq i,j\leq n\}\cup W_0$ and $W_{\beta}=W\setminus W_{\alpha}$. Notice that $W_{\alpha}$ is 
a set of representatives for the transformations of $\ODP_n$ of rank less than or equal to one and $x_n^{n-1}$ is a factor of each word of $W_{\alpha}$.
Let $\overline{W}=W\cup \{wh\mid w\in W_{\beta}\}$. 
Since also  
$
|\overline{W}|=|W|+|W_{\beta}|=2|W|-|W_{\alpha}|=2|\ODP_n|-(n^2+1)=|\DP_n| 
$, 
by Theorem \ref{overpresentation}, we may conclude that the monoid $\DP_n$ is defined by the presentation 
$\langle B\mid \overline{R} \rangle$. 

On the other hand, it is easy to show that relations $(\NR_1)$ are a consequence of $(\NR_0)$ and 
\begin{enumerate}
\item[$(\overline{\NR}_1)$] 
$hx_n=x_{n-1}h\quad\text{and}\quad hx_i=x_{n-i-1}h$,~ 
$1\le i\le\lfloor\frac{n-1}{2}\rfloor$; 
\end{enumerate}
Moreover, within the context of relations $(\NR_0)$ and $(\NR_1)$: 
\begin{enumerate}

\item Relations $(R_1)$ are equivalent to relations 
\begin{enumerate}
\item[$(\overline{R}_1)$] $x_i^2=x_i$,~ $1\leq i\leq\lfloor\frac{n-1}{2}\rfloor$;
\end{enumerate}

\item Relations $(R_3)$ are equivalent to relations $(R_4)$;  

\item Relations $(R_5)$ and $(R_7)$ are (both) equivalent to relation 
\begin{enumerate}
\item[$(\overline{R}_7)$] $x_n^2hx_nh=hx_1hx_n$; 
\end{enumerate}

\item Relations $(R_6)$ and $(R_8)$ are (both) equivalent to relation 
\begin{enumerate}
\item[$(\overline{R}_8)$] $hx_nhx_n^2=x_nx_1$; 
\end{enumerate}

\item Relations $(R_9)$ and $(R_{10})$ are (both) equivalent to relation 
\begin{enumerate}
\item[$(\overline{R}_{10})$] $x_nhx_nhx_n=x_n$. 
\end{enumerate}
\end{enumerate}

Therefore, if we denote by $U$ the set of monoid relations over $B$ 
formed by all relations from 
$$
(\overline{R}_1),~ (R_2),~ (R_4),~ (\overline{R}_7),~ 
(\overline{R}_8),~ (\overline{R}_{10}),~ (R_{11}),~ 
(R_{12}),~ (\NR_0),~ (\overline{\NR}_1)~ \text{and}~(\NR_2), 
$$
we have: 

\begin{theorem}
The monoid $\DP_n$ is defined by the presentation $\langle B\mid U \rangle$, on $n+1$ generators and $\frac{1}{2}(n^2-n+13-(-1)^n)$ relations.
\end{theorem}

We finish this section, and the paper, by \textit{removing} superfluous generators from the above presentation of $\DP_n$, i.e. since $x_{n-1}=hx_nh$ and 
$x_{n-i-1}=hx_ih$, for $1\le i\le \lfloor \frac{n-1}{2} \rfloor$, we remove all  letters/generators $x_i$, for $\frac{n-1}{2} < i\le n-1$, replace all of their occurrences by the previous expressions in all relations of $U$ and remove all  trivial relations and all relations clearly deductible from others obtained in the process.   

First, we turn our attention to relations $(R_2)$. Since $1\le i<j\le n-2$ may be decompose in 
$$
1\le i<j\le \lfloor \frac{n-1}{2}\rfloor
\quad\text{or}\quad 
1\le i\le \lfloor \frac{n-1}{2} \rfloor <j\le n-2
\quad\text{or}\quad
\lfloor \frac{n-1}{2} \rfloor<i<j\le n-2~, 
$$ 
and, for the case $\lfloor \frac{n-1}{2} \rfloor<i<j\le n-2$, we have $x_ix_j=x_jx_i$ if and only if $(hx_{n-i-1}h)(hx_{n-j-1}h)=(hx_{n-j-1}h)(hx_{n-i-1}h)$ if and only if 
$x_{n-i-1}x_{n-j-1}=x_{n-j-1}x_{n-i-1}$, and $1\le n-j-1<n-i-1\le  \lfloor \frac{n-1}{2} \rfloor$, then we obtain the following relations from $(R_2)$: 
\begin{enumerate}
\item[$(\overline{R}_2)$] $x_ix_j=x_jx_i$,~ $1\leq i<j\leq \lfloor \frac{n-1}{2} \rfloor$,~ and 
~$x_ihx_{n-j-1}h=hx_{n-j-1}hx_i$,~ $1\leq i\leq \lfloor\frac{n-1}{2} \rfloor <j\leq n-2$. 
\end{enumerate}

Now, we consider the relations $(R_4)$: $x_nx_{i+1}=x_ix_n$, with $1\le i\le n-3$. For $i=\lfloor\frac{n-1}{2}\rfloor$, we get the relation 
$x_nhx_{\lceil\frac{n-3}{2}\rceil}h=x_{\lfloor\frac{n-1}{2}\rfloor}x_n$ and, for 
$\lfloor\frac{n-1}{2}\rfloor< i\le n-3$, we have 
$x_nhx_{n-i-2}h=hx_{n-i-1}hx_n$, i.e. $x_nhx_jh=hx_{j+1}hx_n$, 
with $1\le j\le \lceil\frac{n-5}{2}\rceil$. Hence, we obtain the following relations from $(R_4)$:
\begin{enumerate}
\item[$(\overline{R}_4)$] $x_nx_{i+1}=x_ix_n$,~ 
$1\le i\le \lfloor\frac{n-3}{2}\rfloor$,~  
$x_nhx_{\lceil\frac{n-3}{2}\rceil}h=x_{\lfloor\frac{n-1}{2}\rfloor}x_n$ ~and~ 
$x_nhx_ih=hx_{i+1}hx_n$,~ $1\le i\le \lceil\frac{n-5}{2}\rceil$. 
\end{enumerate}

Next, from $(R_{11})$ and $(\NR_2)$, we obtain 
\begin{enumerate}
\item[$(\overline{R}_{11})$] 
$x_n^n=x_{1}\cdots x_{\lfloor\frac{n-1}{2}\rfloor}hx_{\lceil\frac{n-3}{2} \rceil}\cdots x_1x_nx_1h$, 


\item[$(\overline{\NR}_2)$] $x_n^{n-1}h=x_nhx_nx_1\cdots 
x_{\lceil\frac{n-3}{2} \rceil}hx_{\lfloor \frac{n-1}{2} \rfloor}\cdots x_1$, 
\end{enumerate}
respectively. 

Finally, it is easy to show that, from $(\overline{\NR}_1)$, only for $n$ odd we obtain non trivial relations and, in fact, a unique one, namely 
\begin{enumerate}
\item[$(\widehat{\NR}_1)$]  $hx_{\frac{n-1}{2}}=x_{\frac{n-1}{2}}h$,~ if $n$ is odd. 
\end{enumerate}

Let us assume that $(\widehat{\NR}_1)$ denotes the empty set for $n$ even. 

Denote by $V$ the set of monoid relations over $B$ 
formed by all relations from 
$$
(\overline{R}_1),~ (\overline{R}_2),~ (\overline{R}_4),~ (\overline{R}_7),~ 
(\overline{R}_8),~ (\overline{R}_{10}),~ (\overline{R}_{11}),~ 
(R_{12}),~ (\NR_0),~ (\widehat{\NR}_1)~ \text{and}~(\overline{\NR}_2). 
$$
Thus, we have: 

\begin{theorem}
The monoid $\DP_n$ is defined by the presentation $\langle C\mid V \rangle$, on $\lfloor \frac{n+3}{2} \rfloor$ generators and $\frac{3}{8}n^2+\frac{45}{8}$  relations, for $n$ odd, and $\frac{3}{8}n^2-\frac{1}{4}n+5$ relations, for $n$ even.
\end{theorem}


\bigskip 

\lastpage 

\end{document}